\documentclass[12pt,twoside]{amsart} 
 \title[On dual F-signature]{On dual F-signature}

\author[A. Sannai]{Akiyoshi Sannai}

\address{Graduate School of Mathematics, Nagoya University Chikusa-ku, Nagoya 464–8602 JAPAN.}

\email{x12004h@math.nagoya-u.ac.jp}

\date{\today}

\usepackage{pxfonts}

\usepackage{latexsym}
\usepackage{amsmath}
\usepackage{amssymb}
\usepackage{amsthm}
\usepackage{amscd}
\usepackage{enumerate} 
\usepackage{amssymb} 
\usepackage{mathrsfs}          
\usepackage[all]{xy}
\usepackage{color}

\usepackage[dvipdfm,bookmarks=true,bookmarksnumbered=true,%
 pdftitle={},%
 pdfsubject={},%
 pdfauthor={},%
 pdfkeywords={TeX; dvipdfmx; hyperref; color;},%
 colorlinks=true]{hyperref}

\newtheorem{thm}{Theorem}[section]

\newtheorem{prop}[thm]{Proposition}

\newtheorem{lem}[thm]{Lemma}

\newtheorem{cor}[thm]{Corollary}

\newtheorem{conj}[thm]{Conjecture}

\newtheorem{ques}[thm]{Question}

\theoremstyle{definition}

\newtheorem{defi}[thm]{Definition}

\newtheorem{eg}[thm]{Example}

\newtheorem*{ack}{Acknowledgments}

\theoremstyle{definition}
\newtheorem{rem}[thm]{Remark}

\begin{document}
\bibliographystyle{amsalpha+}
 
 \maketitle
 
\begin{abstract}

We define the dual F-signature of modules, which is equivalent to the F-signature if the module is the base ring.
By using this invariant, we give characterizations of regular, F-regular, F-rational, and Gorenstein singularities.
\end{abstract}

\section{Introduction}
Let $R$ be complete d-dimensional reduced Noetherian local ring with prime characteristic $p > 0$ and perfect residue field $k = k^p$. There is the Frobenius map $F : R \rightarrow R$ sending $r$ to $r^p$. For $e  \in \mathbb{N}$ the inclusion $R \subseteq R^{1/p^e} $into the corresponding ring of $p^e$-th roots of elements of R is naturally identified with the $e$-th iterate of the Frobenius endomorphism. The $R$-module $R^{1/p^e}$ has important information about the singularity of $R$. Write $R^{1/p^e} = R^{\oplus a^e} \oplus M_e$ as $R$-modules where $M_e$ has no free direct summands. The number $a_e$ is called the $e$-$th$ $Frobenius$ $splitting$ $number$ $of$ $R$. Kunz showed that $a_e=p^{ed}$ holds for some $e$ if and only if $R$ is regular \cite{KunzCharacterizationsOfRegularLocalRings}. This result also tells us that the ratio of  the rank of the free direct summand $a_q$ to the rank of $R^{1/p^e}=p^{ed}$ reflects the distance to the regularity. We define the $F$-signature by the asymptotic behavior of the sequence $\{a_e/p^{de}\}$, namely, $s(R) = \lim_{e \to \infty} \frac{a_{e}}{p^{ed}}$. The $F$-signature first appeared implicitly
in the work of K. Smith and M. Van den Bergh
\cite{SmithVanDenBerghSimplicityOfDiff}, and its formal study was started in \cite{HunekeLeuschkeTwoTheoremsAboutMaximal}.  
Though the existence of the limit had been open for several years, K.Tucker showed the existence in general \cite{TuckerFSignatureExists}. \\
  C. Huneke and G. J. Leuschke showed that $s(R) \le 1$ and the equality holds if and only if $R$ is regular\cite{HunekeLeuschkeTwoTheoremsAboutMaximal}. In their paper, they believed that $s(R)$  characterizes  $F$-rationality of $R$ by its positivity. But I. Aberbach and G. J. Leuschke  showed $s(R)>0$ if and only if $R$  is strongly $F$-regular. In this paper, we prove that what they believe is the right idea, namely, $F$-rationality is characterized by how the canonical module relates to iterated Frobenius powers. To do this, we extend this invariant $s(R)$ to modules. 
  
\begin{defi}
Let $(R, \mathfrak{m}, k)$ be a reduced $F$-finite local ring  of characteristic $p>0$ and $M$ be an $R$-module. For each natural number $e$, put $q=p^e, \alpha = log_p[k:k^p]$, and $ b_{q} = \max \{  \ n  \ |  \  ^\exists  F_*^eM \twoheadrightarrow  M^n  \ \}$
and define  \\
\[s(M) = \lim \sup_{e \to \infty}\frac{b_q}{q^{dimR+\alpha}}
\]

We call $b_q$ the $F$-$surjective$ $number$ $of$ $M$ and call $s(M)$ $the$ $dual$ $F$-$signature$ $of$ $M$.
\end{defi}
The point is that though the $F$-signature depends only on the $R$-module structure of $F_*^eR$, the dual $F$-signature depends not on the $R$-module structure of $F_*^eM$ but on the relative structure of $F_*^eM$ and $M$. It is easy to see the dual $F$-signature of $R$ is equivalent to $F$-signature.  By using dual $F$-signature, we have the characterization of the singularities of $R$.
\begin{thm}\label{main}
Let $(R, \mathfrak{m}, k)$ be a reduced F-finite Cohen-Macaulay local ring of characteristic $p>0$ and $\omega_R$ be the canonical module, then the following hold:
\begin{enumerate}
\item The following are equivalent.
\begin{enumerate}
\item$R$ is regular. 
\item$s(R)=1$
\item$s(\omega_R)=1$
\end{enumerate}
\item $R$ is strongly F-regular if and only if $s(R)>0$
\item Assume $k$ is infinite, then $R$ is F-rational if and only if $s(\omega_R)>0$
\item $s(\omega_R)\ge s(R)$
\item Futhermore, assume $s(\omega_R)>0$, then\\
$R$ is Gorenstein if and only if $s(\omega_R)=s(R)$
\end{enumerate}

 \end{thm}
\begin{rem}
The F-finiteness of $R$ implies the existence of the dualizing complex by Gabber \cite{gabber}. Therefore, there always exists a canonical module under the assumption in the theorem above.
\end{rem}
\begin{rem}
Though the theorem is only for absolute version, we can also generalize $F$-signature of pair, $F$-splitting ratio, and $s$-dimension in the dual situation. See the section 4.
\end{rem}
\begin{ack}
The author would like to express his gratitude
to Professor Mitsuyasu Hashimoto for valuable discussions and helpful comments. The author would like to thank Professor Vasudevan Srinivas and Professor Kenichi Yoshida for examples about the dual $F$-signature. The author also would like to thank Professor Yujiro Kawamata for his suggestion about pair of the dual $F$-signature.  The author is partially supported by JSPS postdoctoral fellow.

\end{ack}

\section{preliminary}
We first review the theory of Hilbert-Kunz multiplicity. Let $(R,\mathfrak{m}, k)$ be a local Noetherian ring of prime
characteristic $p$, $I$ an $\mathfrak{m}$-primary ideal and $M$ an finitely generated $R$-module. The
 Hilbert--Kunz function of $M$ along $I$ is the function taking
an integer $n$ to the length of $R/I^{[p^n]}\otimes M$, where
$I^{[p^n]}$ is the ideal generated by all the $p^n$th powers of
elements of $I$. The Hilbert--Kunz multiplicity of $M$ along $I$, denoted
$e_{HK}(I, M)$, is ${\displaystyle \lim_{q = p^n
\rightarrow \infty}}
\frac{l(R/I^{[q]}\otimes M)}{q^d}$,
where $l(M)$ is the length of $M$. The limit always exists by the following theorem .
\begin{thm} \cite{MonskyHKFunction}
Suppose $(R,\mathfrak{m},k)$ is a local ring of dimension
$d$ and characteristic $p>0$.  If $I$ is any $\mathfrak{m}$-primary ideal and
$M$ is a finitely generated $R$-module, then
the limit 
\[
e_{HK}(I;M) := \lim_{e\to\infty}\frac{1}{p^{ed}} l_{R}(R/I^{[p^{e}]} _{R}\otimes M)
\]
exists and is called the \emph{Hilbert-Kunz multiplicity of $M$ along $I$}.
\end{thm}

We need the following result which is special case of Theorem~8.17 of
\cite{Hochster-Huneke:1990}:  

\begin{thm}\label{hilbert-kunz} Let $(R,\mathfrak{m}, k)$ be a reduced,
F-finite local ring of   
characteristic $p>0$. Let $I\subseteq J$ be two $\mathfrak{m}$-primary ideals. Then
$I^* = J^*$ if and only if $e_{HK}(I, R) = e_{HK}(J, R)$. (Here $I^*$ denotes
the
tight closure of $I$.)
\end{thm}

We note that the assumption concerning test elements in \cite[Theorem
8.17]{Hochster-Huneke:1990} 
is satisfied in this case since the ring is reduced, local and $F$-finiteness implies excellentness .
See \cite[Theorem 6.1]{Hochster-Huneke:1994}.\\
We also need the following proposition in \cite{HunekeTightClosureBook}.
\begin{prop}\label{HK}
Let ($R$, $\mathfrak{m}$, k) be a local Noetherian ring of characteristic $p>0$ and dimension $d$.
Let $M$, $N$ be finitely generated $R$-modules, and let $I$ be an $\mathfrak{m}$-primary ideal.
\begin{enumerate}
\item If dim$M<d$, then $l(M/I^{[q]}M)=O(q^{d-1})$ and thus $e_{HK}(I, M) =0$.
\item Let W be the complement of the set of minimal primes $Q$ such that dim$(R/Q)=d$. If $M_W \cong N_W$,\begin{eqnarray*}
\left|(M/I^{[q]}M)-l(N/I^{[q]}N)\right|=O(q^{d-1})
\end{eqnarray*}
In paticular, $e_{HK}(I, M) =e_{HK}(I, N)$.
\end{enumerate}
\end{prop}
For the proofs, we need to consider about the minimal difference of Hilbert-Kunz multiplicity along some ideals. Though the minimal difference of Hilbert-Kunz multiplicity along $\mathfrak{m}$-primary ideals (called minimal relative Hilbert-Kunz multiplicity) is introduced by K.-i. Watanabe and K. Yoshida \cite{WatanabeYoshidaMinimalHKmultiplicity}, We need the "parameter ideal" version of minimal relative Hilbert-Kunz multiplicity. The following definition and theorem are due to M. Hochster and Y. Yao \cite{F-rat}.
\begin{defi}
Let $(R, \mathfrak{m}, k)$ be a local ring of prime characteristic $p>0$ and $M$ be a finitely
generated R-module. Define (here s.o.p. stands for system of parameters) 
 \[
r_{R(M)}=\inf \{e_{HK}((\underline{x}),M)-e_{HK}((\underline{x},\Delta),M)|\  \underline{x}\  is\  a\  s.o.p. \ and((\underline{x}):\Delta)=\mathfrak{m}\}
 \]
\end{defi}
\begin{thm}\label{rsig}
 Let $(R, \mathfrak{m}, k)$ be a Noetherian local ring of characteristic p. Suppose $R$ is excellent or there exists a common parameter (weak) test element for $R$ and $\hat{R}$. Then $R$ is F -rational if and only if $r_{R}(R) > 0$.
\end{thm}

We review the $F$-signature of modules defined by Y.Yao\cite{YaoObservationsAboutTheFSignature} and the result by K. Tucker about it \cite{TuckerFSignatureExists}. 

\begin{defi}
Let  $(R, \mathfrak{m}, k)$ be an $F$-finite local ring and $M$ a finitely generated R-module. For each $e \in \mathbb{N}$, put $q=p^e, \alpha = log_p[k:k^p]$, and write $F_*^eM \cong R^{a_q} \oplus M_e$ as left $R$-modules such that $M_e$ has no non-zero free direct summand. In other words, the number $a_e$ is the maximal rank of free direct summand of the left R-module $F_*^eM$.  We define
\[
 s'(M) =\lim_{e \to \infty}\frac{a_q}{q^{dimR+\alpha}}
 \]
\end{defi}

\begin{thm}
\label{tucker}
   Let $(R, \mathfrak{m}, k)$ be a
 $d$-dimensional $F$-finite characteristic $p > 0$ local
  domain and let $M$ be a finitely generated $R$-module. Denote by $a_{e}$ the maximal rank of a free
$R$-module appearing in a direct sum decomposition of
  $F^{e}_{*}M$.  Then  
\[
\lim_{q \to \infty} \frac{a_{q}}{q^{(d+\alpha)}} =s'(M)= rank(M) \cdot
s(R) \, \, .
\]
\end{thm}

\section{The dual $F$-signature}
We define the dual $F$-signature of modules.
\begin{defi}
Let $(R, \mathfrak{m}, k)$ be a reduced $F$-finite local ring  of characteristic $p>0$ and $M$ be an $R$-module. For each natural number $e$, Put $q=p^e, \alpha = log_p[k:k^p]$, and $ b_{q} = \max \{  \ n  \ |  \  ^\exists  F_*^eM \twoheadrightarrow  M^n  \ \}$
and define 
\[
 s(M) =\lim \sup_{e \to \infty}\frac{b_q}{q^{dimR+\alpha}}
 \]
We call $b_q$ $q$-th $F$-surjective number of $M$ and call $s(M)$ the dual $F$-signature of $M$.\\
\end{defi}

\begin{rem}
Since any homomorphism $\phi : F_*^eR \twoheadrightarrow  R^n$ splits, We can easily see that $s(R)$ coincides with the $F$-signature defined by  G. J. Leuschke and C. Huneke.
\end{rem}
Theorem1.2 (2) follows from the following result due to I. M. Aberbach and G. J. Leuschke\cite{AberbachLeuschke}. 
\begin{thm}
Let $(R, \mathfrak{m}, k)$ be a reduced excellent $F$-finite local ring containing a field of characteristic p, let $d=$dim$R$. Then $s(R)$ is positive if and only if $R$ is strongly $F$-regular.
\end{thm}

\begin{prop}\label{zero}
Let $(R, \mathfrak{m}, k)$ be a reduced F-finite local ring and $M$ be a finitely generated $R$-module . Assume dim$M<$dim$R$, then $s(M)=0$.
\end{prop} 
\begin{proof}
Assume we have 
\[
F_*^eM\longrightarrow  M^{b_q} \longrightarrow 0
\]
Tensoring $R/I$, we have
\[
F_*^e(M/I^{[q]}M)\longrightarrow  (M/IM)^{b_q} \longrightarrow 0
\]
Therefore we have
\[
l(F_*^e(M/I^{[q]}M)\geq b_q\cdot l(M/IM)
\]
Dividing $q^{d+\alpha}$ and taking the limit, we have
\[
e_{HK}(I, M) \geq s(M)\cdot l(M/IM)
\]
The left hand side is zero by Theorem \ref{HK} (1).

\end{proof}

\begin{prop}\label{value}
Let $(R, \mathfrak{m}, k)$ be a reduced F-finite local ring and $M$ be a finitely generated $R$-module . Then $0\le s(M)\le1$.
\end{prop}
\begin{proof}
Asuume dim $M<$ dim $R$, then $s(M)=0$ by Proposition \ref{zero}.
We may assume dim $M=$ dim $R$. Therefore there is a minimal prime $\mathfrak{p}$ such that
$M$ has rank at $\mathfrak{p}$. Since the rank of $(F_*^eM)_{\mathfrak{p}}$ is rank $M\cdot q^{d+\alpha}$ and the rank of $(M^{b_q})_{\mathfrak{p}}$ is rank $M\cdot b_q$, we obtain $b_q \le q^{d+\alpha}$. This implies $s(M)\le1$
\end{proof}

\begin{lem}\label{MCM}
Let $(R, \mathfrak{m}, k)$ be a reduced F-finite Cohen-Macaulay local ring and $n$ be a non negative integer. There is a one to one correspondence between non-isomorphic surjective homomorphisms from $F_*^e\omega_R$ to $\omega_R^n$ and non-isomorphic injective homomorphisms from $R^n$ to $F_*^eR$ such that the cokernel is a maximal Cohen-Macaulay module.
\end{lem}
\begin{proof}
Assume we have
\[
0\longrightarrow ker(f)\longrightarrow F_*^e\omega_R\longrightarrow  \omega_R^{b_q} \longrightarrow 0
\]
Since $F_*^e\omega_R$ and $\omega_R$ are maximal Cohen-Macaulay module, the kernel is maximal Cohen-Macaulay module. Taking $\omega_R$-dual, we obtain
\[
0\longrightarrow R^{b_q}\longrightarrow Hom(F_*^e\omega_R, \omega_R)\cong F_*^eR \longrightarrow  K\longrightarrow 0
\]
where $K$ is the $\omega_R$-dual of $ker(f)$. Therefore $K$ is maximal Cohen-Macaulay module. Conversely, assume we have
\[
0\longrightarrow R^{b_q}\longrightarrow F_*^eR \longrightarrow  coker(g)\longrightarrow 0
\]
such that $coker(g)$ is a maximal Cohen-Macaulay module. Taking $\omega_R$-dual, we obtain
\[
0\longrightarrow L\longrightarrow F_*^e\omega_R \longrightarrow  \omega_R^{b_q} \longrightarrow Ext^1_R(coker(g), \omega_R)=0
\]
where $L$ is the $ \omega_R$-dual of $coker(g)$. The last equality follows from Grothendieck duality. 

\end{proof}
\begin{rem}\label{max}
If there is a surjective homomorphism from $ker(f)$ to $\omega_R$, then there is a surjective homomorphism from $F_*^e\omega_R$ to $\omega_R^{b_q+1}$.
\end{rem}
\begin{proof}
Consider the following diagram
\[
\xymatrix{0\ar[r]&ker(f)\ar[r]\ar[d]&F_*^e\omega_R
\ar[r]\ar[d]&\omega_R^{b_q}\ar[r]\ar[d]&0\\
0\ar[r]&\omega_R\ar[r]&L
\ar[r]&\omega_R^{b_q}\ar[r]&0}
\]
where $L$ is the push out of the diagram. Since $L$ is in $Ext^1(\omega_R^{b_q}, \omega_R)=0$, the exact sequence splits. $L$ is isomorphic to  $\omega_R^{b_q+1}$.
\end{proof}

\begin{prop}\label{hikaku}
Let $(R, \mathfrak{m}, k)$ be a reduced F-finite Cohen-Macaulay local ring. Then $s(\omega_R)\geq s(R)$.

\end{prop}
\begin{proof}
Assume we have
\[
0\longrightarrow ker(f)\longrightarrow F_*^eR\longrightarrow  R^{a_q} \longrightarrow 0
\]
Since $R^{b_q}$ is projective, this exact sequence splits. Therefore $ker(f)$ is maximal Cohen-Macaulay module. We also have
\[
0\longrightarrow R^{a_q}\longrightarrow F_*^eR \longrightarrow  ker(f)\longrightarrow 0
\]
and this implies $s(\omega_R)\geq s(R)$ by Lemma \ref{MCM}.

\end{proof}

\begin{thm}\label{reg}
Let $(R, \mathfrak{m}, k)$ be a reduced F-finite Cohen-Macaulay local ring. Then $s(\omega_R)=1$ if and only if $R$ is regular.
\end{thm}
\begin{proof}
If $R$ is regular, then the result follows from Kunz's theorem. Therefore It is enough to prove the converse.
Let $\underline{x}$ be a system of parameters generating a minimal reduction of $\mathfrak{m}$, then $l(\mathfrak{m}/\underline{x})=e(\mathfrak{m})-1$. By the definition of $b_q$ and Lemma \ref{MCM}, we have
\[
0\longrightarrow R^{b_q}\longrightarrow F_*^eR \longrightarrow  K\longrightarrow 0
\]
where $K$ is a Cohen-Macaulay module. By tensoring $R/ \mathfrak{m}$ and $R/(\underline{x})$, we have 

\[
\xymatrix{0\ar[r]&(R/(\underline{x}))^{b_q}\ar[r]\ar[d]&F_*^e(R/(\underline{x})^{[q]})
\ar[r]\ar[d]&K/(\underline{x})K\ar[r]\ar[d]&0\\
Tor^R_1(K, R/\mathfrak{m})\ar[r]&(R/\mathfrak{m})^{b_q}\ar[r]&F_*^e(R/\mathfrak{m}^{[q]})
\ar[r]&K/\mathfrak{m}K\ar[r]&0}
\]
The injectivity of the second map of the first line follows from
\[
Tor^R_1(K, R/(\underline{x}))\cong Ext^{d-1}_R(R/(\underline{x}), K)=0
\]
The vertical maps are surjective. We have
\[
l(F_*^e(R/(\underline{x})^{[q]})-l(F_*^e(R/\mathfrak{m}^{[q]})\\
\geq b_q(l(R/(\underline{x})-l(R/\mathfrak{m}))+
l(K/(\underline{x})K)-l(K/\mathfrak{m}K)\]
\[
\geq b_q(e-1)
 \]
Dividing both side by $q^{d+\alpha}$and taking the limit, we obtain
\[
e_{HK}((\underline{x}), R)-e_{HK}(\mathfrak{m}, R)\geq e-1
\]
 Since $(\underline{x})$ is minimal reduction for $\mathfrak{m}$ and $R$ is Cohen-Macaulay, $e_{HK}((\underline{x}), R)=e(\mathfrak{m}, R)=e$. This gives $e_{HK}(\mathfrak{m}, R)=1$ and $R$ is regular.
\end{proof}

\begin{prop}\label{gor}
Let $(R, \mathfrak{m}, k)$ be a reduced F-finite local ring such that $s(\omega_R)>0$.
Then $R$ is Gorenstein if and only if $s(R)=s(\omega_R)$.
\end{prop}

\begin{proof}
It is enough to show the converse. Since $s(R)=s(\omega_R)>0$, $R$ is $F$-regular.
Assume $R$ is not Gorenstein.
Let
\[
F_*^e\omega_R\cong R^{a_q}\oplus \omega_R^{b_q}\oplus M
\]
be a direct summand deconposition of $F_*^e\omega_R$ such that $M$ has no direct summand of $R$ or $\omega_R$.
We obtain
$
\displaystyle
\ \ \lim_{q \to \infty}\frac{a_q}{q^{d+\alpha}}=s(R)
 $ by Theorem \ref{tucker}. Taking $\omega_R$-dual, we have
\[
F_*^eR \cong R^{b_q}\oplus \omega_R^{a_q}\oplus Hom(M, \omega_R)
\]
This implies$
\displaystyle
\ \ \lim_{q \to \infty}\frac{b_q}{q^{d+\alpha}}=s(R)$. Let us denote the number of minimal generaters of $\omega_R$ by $\mu(\omega_R)$. There is a surjective homomorphism from $R^{\mu(\omega_R)} $ to $\omega_R$.
It follows that $s(\omega_R) \geq \displaystyle\frac{s(R)}{\mu(\omega_R)}+s(R) > s(R)$.
This is contradiction.
\end{proof}

\begin{rem}
We remark the assumption in Theorem \ref{gor} is essential.
Let $R$ be a non-Gorenstein, non-$F$-injective local ring, then $s(R)=s(\omega_R)=0$. But $R$ is not Gorenstein.
\end{rem}

\begin{cor}
Let $(R, \mathfrak{m}, k)$ be a reduced F-finite Cohen-Macaulay local ring. Then $s(R)=1$ if and only if $R$ is regular.
\end{cor}
\begin{proof}
Assume $s(R)=1$, then $s(\omega_R)=1$ by Proposition \ref{hikaku}.
This implies that $R$ is regular by Theorem \ref{reg}.
\end{proof}

\begin{lem}\label{inj}
Let $(R, \mathfrak{m}, k)$ be an Artinian local ring of infinite residue field and $M$ be an $R$-module. Let V be a sub k-vector space of soc(R).
Assume
$l(\Delta M) \ge dimV$ for any $\Delta\in socR$, then there is an R-homomorphism $\phi : R \rightarrow M$ such that  $\phi$  is injective on V.
In paticular, if V=soc(R), then there is an injective R-homomorphism $\phi : R \rightarrow M.$
\end{lem}
\begin{proof}
We prove the claim by induction on dimV.
Assume dim$V$=1, then there is an socle element $\Delta_1$ which generates $V$.
By the assumption, there is an element $m$ in $M$ such that  $\Delta_1 m \neq 0$. This $m$ gives the map from $R$ to $M$. Assume dim$V$=n, then there are socle elements $\Delta_1,...,\Delta_n$ which generate $V$. 
Let $W$ be the $k$-vector space generated by $\Delta_1,...,\Delta_{n-1}$. By induction, there is a R-homomorphism $\phi : R \rightarrow M$ such that  $\phi$  is injective on $W$. We put $\phi (1)=m$. If this map is injective on $V$, there is nothing to prove. So we may assume $\phi$  is not injective on $V$. Then there is an element $\Delta \in V$ such that $\Delta m=0$.  In this case, $\Delta_1,...,\Delta_{n-1}, \Delta$ is a basis of $V$. Since $l(\Delta M) > $dim$W$, there is an element $m'$ in $M$ such that $\Delta_1 m,...,\Delta_{n-1} m, \Delta m'$ are linearly independent in $soc(R)M$. This implies that a minor of the matrix ($\Delta_1 m,...,\Delta_{n-1} m, \Delta m'$)
is not zero. Let us  Take $c \in R$ and put $n=m+cm'$. If $c$ is general, $c$ gives the map from $R$ to $M$ which satisfy the desired condition. To confirm this, we denote the corresponding minor of the matrix $A$ by $minorA$. With this notation,
\[
minor(\Delta_1 n,...,\Delta_{n-1} n, \Delta n)=minor(\Delta_1 m+c\Delta_1m',...,\Delta_{n-1} m+c\Delta_{n-1}m', c\Delta m')
\]
\[
=c\cdot minor(\Delta_1 m,...,\Delta_{n-1} m, \Delta m')+c^2f(c)
\]
where $f$ is a polynomial with a variable.\\ 
Since $minor(\Delta_1 m,...,\Delta_{n-1} m, \Delta m')$ is not zero,  $minor(\Delta_1 n,...,\Delta_{n-1} n, \Delta n)$ is non zero polynomial with the variable $c$. Therefore for general $c$, the determinant is not zero. This implies $\Delta_1 n,...,\Delta_{n-1} n, \Delta n$ is linearly independent and the map induced by $n$ is injective on $V$.

\end{proof}

\begin{lem}\label{matlis}
Let $R$ be a Noetherian local ring and $M$ be an $R$-module of finite length.Denote the injective hull of the residue field by $E$. Then  $l(xM)=l(Hom(xM, E))=l(x(Hom(M, E))$.  
\end{lem}
\begin{proof}
The first equality follows from the Matlis duality. It is enough to show $Hom(xM, E)$ is isomorphic to $x(Hom(M, E))$. From the $R$-module structure of $Hom(M, E)$, we can regard $x(Hom(M, E))$ as submodule of $Hom(xM, E)$. Conversely, if we take $\phi \in Hom(xM, E)$, we can extend $\phi$ to $\widetilde{\phi} \in Hom(M, E)$ since $E$ is injective module. This implies that the inclusion is surjective. 

\end{proof}

\begin{lem}\label{length}
Let $R$ be a Noetherian local ring and $M$ be an $R$-module. Let $\underline{x}$ be a system of parameters. There is a natural number $c$ such that
$c \ge l(Tor^R_1(\omega_R, R/(\underline{x}, \Delta)))$ for any $\Delta\in SocR/(\underline{x})$.

\end{lem}

\begin{proof}
Let $X_{.}$ be a free resolution of $\omega_R$. Then there is a natural surjection from $X_1\otimes R/(\underline{x})$ to $X_1\otimes R/(\underline{x}, \Delta)$. Since $Tor^R_1(\omega_R, R/(\underline{x}, \Delta))$ is sub-quotient of $X_1\otimes R/(\underline{x}, \Delta)$, $l(X_1\otimes R/(\underline{x})) \ge l(Tor^R_1(\omega_R, R/(\underline{x}, \Delta)))$ holds.

\end{proof}
\begin{thm}
Let $(R, \mathfrak{m})$ be a reduced F-finite Cohen-Macaulay local ring. Then $s(\omega_R)>0$ if and only if $R$ is F-rational.
\end{thm}
\begin{proof}

Let $\underline{x}$ be arbitrary system of parameters and $\Delta$ be a element in the socle of $R/\underline{x}$. By the definition of $b_q$, we have\\
\[
0\longrightarrow K\longrightarrow F_*^e\omega_R \longrightarrow  \omega_R^{b_q} \longrightarrow 0
\]
where $K$ is the kernel of the surjective map. Tensoring $R/(\underline{x})$ and $R/(\underline{x}, \Delta)$, we have

\[
{\tiny
\xymatrix{0\ar[r]&K/(\underline{x})K\ar[r]\ar[d]&F_*^e(\omega_R/(\underline{x})^{[q]}\omega_R)
\ar[r]\ar[d]&(\omega_R/(\underline{x})\omega_R)^{b_q}\ar[r]\ar[d]&0\\
Tor^R_1(\omega_R, R/(\underline{x}, \Delta))^{b_q}\ar[r]&K/(\underline{x}, \Delta)K\ar[r]&F_*^e(\omega_R/(\underline{x}, \Delta)^{[q]}\omega_R)
\ar[r]&(\omega_R/(\underline{x}, \Delta)\omega_R)^{b_q}\ar[r]&0}
}
\]
The injectivity of the first map of the first line follows from
\[
Tor^R_1(\omega_R, R/(\underline{x}))\cong Ext^{d-1}_R(R/(\underline{x}), \omega_R)=0
\]
The vertical maps are surjective. We have
\[
l(F_*^e(\omega_R/(\underline{x})^{[q]}\omega_R))-l(F_*^e(\omega_R/(\underline{x}, \Delta)^{[q]}\omega_R))\geq 
\]
\[
b_q(l(\omega_R/(\underline{x})\omega_R)-l(\omega_R/(\underline{x}, \Delta)\omega_R))+
l(K/(\underline{x})K)-l(K/(\underline{x}, \Delta)K))
\geq b_q
 \]
Dividing both side by $q^{d+\alpha}$and taking the limit, we obtain
\[e_{HK}((\underline{x}), R)-e_{HK}((\underline{x}, \Delta), R)\geq s(\omega_R)
\]
Assume $R$ is not F-rational, then there exists $\Delta$ such that $\Delta$ is in the tight closure of $(\underline{x})$. This implies  $s(\omega_R)=0$.
Conversely, assume $R$ is F-rational. From the diagram, we have
\[
b_q(l(\omega_R/(\underline{x})\omega_R)-l(\omega_R/(\underline{x}, \Delta)\omega_R)+l(Tor^R_1(\omega_R, R/(\underline{x}, \Delta)))+
l(K/(\underline{x})K)-l(K/(\underline{x}, \Delta)K)
\]
\[\geq 
l(F_*^e(\omega_R/(\underline{x})^{[q]}\omega_R))-l(F_*^e(\omega_R/(\underline{x}, \Delta)^{[q]}\omega_R))
\]
If $s(\omega_R)=0$, then this means the order of $b_q$ is less than $q^{d+\alpha}$. Since $R$ is F-rational, the order of the right hand side of the inequality is $q^{d+\alpha}$. Lemma \ref{length} and Theorem \ref{rsig} implies that there is a $q$ such that for all $\Delta \in soc(R)$,  $l(\Delta(K/(\underline{x})K))=l(K/(\underline{x})K)-l(K/(\underline{x}, \Delta)K)$ is bigger than $l(soc(R/\underline{x}))$. By Lemma \ref{matlis} and Lemma \ref{inj}, there is an injective homomorphism $\phi : R/(\underline{x}) \rightarrow Hom(K/(\underline{x})K, E(k))$. This implies that there is an surjective homomorphism $\phi : K \rightarrow \omega_R$ by Lemma \ref{MCM}. This contradicts to the maximality of $b_q$ by the Remark \ref{max}.
 
\end{proof}

\begin{eg}\label{ver}
Let $R = k[[x^n, x^{n-1}y, \ldots, y^n]]$, the
$n$-th Veronese
subring of $k[[x,y]]$, where $k$ is a perfect field of positive
characteristic $p$.  Assume that $n\geq 2$ and $p\!\!\!\not| n$.  Then
$R$ has finite CM type.  The indecomposable
nonfree
MCM $R$-modules are the fractional ideals $I_1 = (x,y)$, $I_2 = (x^2,
xy,
y^2)$, \ldots, $I_{n-1} = (x^{n-1}, x^{n-2}y, \ldots, y^{n-1})$.  Denote $R$ also by $I_0$.
Then we can decompose $F_*I_l$ by using these modules.
\[
F_*^eI_l\cong \oplus_{i=0}^{n-1}I_i^{a_{i, l}}
\]
We can also show the order of $a_{i, l}$'s are same and equal to $q^2/n$. 
Furthermore, we can easily show the following things.
\begin{enumerate}
\item If $l<k$, then any $R$-module homomorphism $f$ : $I_k \rightarrow I_l$ factor through $\mathfrak{m}I_l$.
\item If $l>k$, then there is no surjective homomorphism $g$: $I_k\rightarrow I_l$
\item If $l>k$, then there exists a surjective homomorphism $h$: $I_k\oplus I_{l-k-1} \rightarrow I_l$

\end{enumerate}
We compute $s(I_l)$. By (1), there is no contribution from $I_k$ $(l<k)$.
Since there is a natural surjection $id: I_l \rightarrow I_l$, the contribution from $I_l$ is one to one. 
By (2) and (3), the contribution from $I_k$ $(l>k)$ is two to one.
Therefore,
\[
s(I_l)=\frac{1}{n}+\frac{l}{2}\cdot \frac{1}{n}=\frac{l+2}{2n}
\]
In paticular, $s(R)=s(I_0)=1/n, s(\omega_R)=s(I_{n-2})=1/2$.
$R$ is Gorenstein if and only if $n=2$ by Proposition \ref{gor}.

\end{eg}

\section{Questions}
In this section, we define $the$ $dual$ $F$-$signature$ $of$ $pairs$ and  $the$ $F$-$surjective$ $ratio$  to give questions about these.\\
The $F$-signature of pairs was introduced by M. Blickle, K. Schwede and K.Tucker in \cite{BlickleSchwedeTuckerFsigPairs} to solve the conjecture about the $F$-splitting ratio by M. Aberbach and F. Enescu. Firstly, We define $the$ $F$-$surjective$ $ratio$ which is a generalization of $the$ $F$-$splitting$ $ratio$ defined by M. Aberbach and F. Enescu \cite{AberbachEnescuStructureOfFPure}. 

\begin{defi}
Let $(R, \mathfrak{m}, k)$ be a reduced $F$-finite local ring  of characteristic $p>0$ and $M$ be an $R$-module. For each natural number $e$, Put $q=p^e, \alpha = log_p[k:k^p]$, and $ b_{q} = \max \{  \ n  \ |  \  ^\exists  F_*^eM \twoheadrightarrow  M^n  \ \}$
and for each  $i\in \mathbb{N}$ define 
\[
r_i(M) =\lim \sup_{e \to \infty}\frac{b_q}{q^{i+\alpha}}
 \]
Consider about the set 
\[
D=\{r_i(M)| i\in \mathbb{N}\} \in \mathbb{R} \cup \{\infty\}
\]
We define $the$ $F$-$surjective$ $ratio$ by
\[
r(M)= \max\{ D-\{\infty\}\}
\]
We call the minimal number $i$ such that $r_i(M)=r(M)$ holds $F$-surjective dimension of $M$.
\end{defi}
\begin{eg}\label{rat}
Let $S_n=k[[x_1,x_2,\dots ,x_n]]$ be a formal power series ring of $n$ variables.
There are natural surjections $\phi_i: S_n \rightarrow S_i=k[[x_1,x_2,\dots ,x_i]]$ for $i\le n$.
We regard $S_i$ as $S_n$-module by using these maps. Then the $F$-surjective dimension of $S_i$ is i and $r(S_i)=1$. 

\end{eg}

We can easliy see that $r(R)$ coincide with the $F$-splitting ratio defined by M. Aberbach and F. Enescu. M. Blickle, K. Schwede and K.Tucker proved the positivity of the $F$-splitting ratio characterize the $F$-purity of $R$ in \cite{BlickleSchwedeTuckerFsigPairs}. The conjecture is the same thing holds for the $F$-surjective ratio.

\begin{conj} 
Let $(R, \mathfrak{m}, k)$ be a reduced Cohen-Macaulay $F$-finite local ring  of characteristic $p>0$ and $\omega_R$ be a canonical module. Then $R$ is $F$-injective if and only if $r(\omega_R)$ is positive.
\end{conj} 

The case of $F$-splitting ratio was solved by considering about the pair of $F$-signature.
Not only for this application, it is important to consider about $the$ $dual$ $F$-$signature$ $of$ $pairs$.

\begin{defi}
Let $(R, \mathfrak{m}, k)$ be a reduced $F$-finite local ring  of characteristic $p>0$ and $M$ an $R$-module. Let $\mathscr{D}_M$ be a Cartier sub algebra of $M$. For each natural number $e$, Put $q=p^e, \alpha = log_p[k:k^p]$, and $ b_{q} = \max \{  \ n  \ |  \  ^\exists\phi:  F_*^eM \twoheadrightarrow  M^n, pr_i\cdot \phi \in \mathscr{D}_M $ for $ $any $ i \ \}$
and define 
\[
 s(M, \mathscr{D}_M) =\lim \sup_{e \to \infty}\frac{b_q}{q^{dimR+\alpha}}
 \]
We call $b_q$ $q$-th $F$-surjective number of $M$ along $\mathscr{D}_M$ and call $s(M, \mathscr{D}_M)$ the dual $F$-signature of the pair $M$ and $\mathscr{D}_M$.\\

\end{defi}

\begin{eg}
Let $(R, \mathfrak{m}, k)$ be a reduced $F$-finite local ring  of characteristic $p>0$ and $M$ be a finitely generated $R$-module. We denote $ \mathscr{D}^{sp}_{M}$ as the Cartier multiplicative closed sub set  such that all homogeneous elements consist of split-surjective maps. Then  $s(M, \mathscr{D}^{sp}_{M})$ is equivalent to the generalized $F$-signature of $M$ defined in \cite{HN}. More concretely, in the Example \ref{ver}, $s(I_l, \mathscr{D}^{sp}_{I_l})=1/n$.
\end{eg}

\begin{rem}
By using the canonical duality, we can easily check $s(R)=s(\omega_R,  \mathscr{D}^{sp}_{\omega_R})$.

\end{rem}

\begin{ques}
Can we generalize the results in Theorem \ref{main} by using the dual $F$-signature of pairs? 

\end{ques}

\end{document}